\documentclass[a4paper, 12pt, twoside]{article} 

	\usepackage{verbatim}

   \usepackage[authoryear]{natbib}
	\usepackage[left=3cm,right=2.5cm,bottom=3.5cm,top=2.5cm]{geometry} 
	\usepackage{amsfonts}
	\usepackage[utf8]{inputenc} 
	\usepackage[english]{babel} 
	\usepackage{amssymb}        
	\usepackage{amsmath}        
	\usepackage{amsthm}         
	\usepackage{mathtools}      
	
    	\usepackage{enumitem}       
    	\usepackage{xcolor}         
    	\usepackage{fancyhdr}       
    	\usepackage{setspace}       
    \usepackage{bbm} 
    \usepackage{hyperref} 
    
	\usepackage{graphicx} 
	\usepackage{multirow}
	\usepackage{diagbox}
	\usepackage{hhline}
	\usepackage{caption}        

	\usepackage{subcaption}
    	\usepackage{makecell}       
     
\usepackage{algorithm}
\usepackage{algpseudocode}
    
	\DeclareMathOperator*{\argmax}{argmax} 
    
     	\newcommand{\tod}{\xrightarrow{d}}
	\newcommand{\topr}{\xrightarrow{\mathbb{P}}}
     	\newcommand{\ton}{\xrightarrow{n \to \infty}}
	\newcommand{\sgn}{\operatorname{sgn}}

	\newcommand{\R}{\mathbb{R}}

		\newtheoremstyle{Format1}
		{\topsep}   
		{\topsep}   
		{\normalfont}  
		{0pt}       
		{\bfseries} 
		   {\newline}  
		{5pt plus 1pt minus 1pt} 
		{}          

		\theoremstyle{Format1} 
		\newtheorem{Def}{Definition}[section]       
        \newtheorem{Lem}[Def]{Lemma}                

        \newtheorem{Thm}[Def]{Theorem}

	\newtheorem{ex}[Def]{Example}

        \DeclarePairedDelimiter\abs{\lvert}{\rvert}%
        \DeclarePairedDelimiter\norm{\lVert}{\rVert}%
        
        \makeatletter
        \let\oldabs\abs
        \def\abs{\@ifstar{\oldabs}{\oldabs*}}
        \let\oldnorm\norm
        \def\norm{\@ifstar{\oldnorm}{\oldnorm*}}
        \makeatother
        
        \hypersetup{
            colorlinks,
            citecolor=black,
            filecolor=black,
            linkcolor=black,
            urlcolor=black
        }

	\title{Testing equivalence of multinomial distributions -   a constrained bootstrap approach }
\author{
  {\normalsize Patrick Bastian, Holger Dette,  Lukas Koletzko} \\
{\normalsize  Ruhr-Universit\"at Bochum} \\
{\normalsize  Fakult\"at f\"ur Mathematik} \\
{\normalsize  44780 Bochum, Germany}
}

\date{}

\begin{document}

\maketitle

        \begin{abstract}
    In this paper we develop a  novel bootstrap test for the comparison of two multinomial distributions. The two distributions are called {\it equivalent} or {\it similar} if a  norm of the difference between the class probabilities is  smaller than a given threshold. In contrast to most of the literature our approach does not require differentiability of the norm and is in particular applicable for the maximum- and $L^1$-norm.
        \end{abstract}
\section{Introduction}
        	\label{sec1} 
  \def\theequation{1.\arabic{equation}}	
  \setcounter{equation}{0}

In recent years there has been considerable interest in testing equivalence hypotheses regarding the parameters of a multinomial distribution \citep[see][]{wellek2010testing, ostr2017,OSTROVSKI201861,Alba-Fernndez2023}. Most authors concentrate on the one sample problem, where $X = (X_1,  \ldots  , X_k)^\top  $
denotes a $k$-dimensional multinomial distributed vector with $n_1$ trials and   success probabilities $p= (p_1,  \ldots  p_k)^\top $ (with $\sum_{i=1}^k p_i=1$), that is $X
\sim  {\rm Multi}(n_1, p_1,  \ldots  , p_k)$,  and are interested in testing the hypothesis
\begin{align}
                \label{hyp}
                    H_0: d(p, p^{(0)} ) \geq \epsilon \quad vs. \quad H_1: d(p,  p^{(0)}) < \epsilon,
                \end{align}
                where  $d$ is a distance measure between two probability distributions on the set $\{1, \ldots , k\} $,
                $  p^{(0)} =(p_1^{(0)},  \ldots  p_k^{(0)})^\top  $ is a given  vector of probabilities 
                ctor with $n_1$ trials and   success probabilities $p= (p_1,  \ldots  p_k)^\top $ (with $\sum_{i=1}^k p_i^{(0)}=1$) and  $\epsilon >0$ is prespecified threshold.

                These papers mainly differ in the distances that are used in \eqref{hyp} and in the way how the quantiles of a test, which rejects the null hypothesis for small values of an appropriate estimate
                of $d(p, p^{(0)} ) $, are computed. 
            \cite{wellek2010testing} considered the Euclidean distance, \cite{ostr2017} a smoothed total variation distance, \cite{OSTROVSKI201861}  continuously differentiable distances
            in general  and \cite{Alba-Fernndez2023} differentiable $\phi$-divergences. Quantiles are either derived by asymptotic theory (using the asymptotic normality of the  estimate  $\hat p$ and the delta method) or by different variants of the bootstrap. 
            We also mention the work of \cite{frey2009exact}, who proposed to measure the distance between $p$ and $p^{(0)}$ by the maximum deviation between the cumulative distributions. However, as pointed out by this author such an approach is not invariant with respect to permutation of the classes. A common feature of all papers on testing the hypotheses \eqref{hyp} consists in the fact that the distance $d$ is differentiable. On the one hand this simplifies the application of the delta method and  yields an asymptotic normal distribution of the  statistic $ T_{n_1}= \sqrt{n_1} ( d(\hat p, p^{(0)})  - d(p, p^{(0)} ) )$. On the other hand, if the sample size is small and resampling is used to calculate critical values, the results  of \cite{fang2019inference} indicate that differentiability is sufficient and necessary  to prove  that the bootstrap statistic $
            T_{n_1}^* = \sqrt{n_1} ( d(\hat p^*, p^{(0)})  - d( \hat p, p^{(0)} ) ))$ has (conditionally on the data) the same limit distribution as $T_{n_1}$, which is the commonly used argument for proving bootstrap consistency.
                
{\bf Our contribution:}  The present paper contributes to  this discussion and considers the problem of equivalence testing for a general class of distances, which are not necessarily differentiable. This class includes 
the important cases of the maximum deviation and $L^1$-distance, which are not  differentiable functions of $\mathbb{R}^k$. 
Despite the negative results of \cite{fang2019inference} we will develop a constrained bootstrap test for the hypotheses \eqref{hyp}
and prove its validity. The price that we pay for this result consists in the fact that in the case of non-differentiabiltiy the tests are in some cases slightly conservative.
We also illustrate the validity of our approach  for finite sample sizes by means of small simulation study. In contrast to the literature we consider the two-sample problem.
However, the results can easily be transferred to the one-sample problem
as considered in the cited references.
We emphasize that our approach also provides an alternative test for the distances which have been considered so far in the literature. In other words: the test  is applicable, independently of the distance under consideration as long as this distance is directionally Hadamard  differentiable in the sense of Definition 2.1 in \cite{fang2019inference} with a sublinear  derivative, and these assumptions are satisfied for all distances which have been considered so far in the literature.

\section{Comparing two multinomial distributions}

\label{sec2} 
  \def\theequation{2.\arabic{equation}}	
    \setcounter{equation}{0}
  
         Let $X = (X_1,  \ldots  , X_k) \sim 
         {\rm Multi}(n_1, p_1,  \ldots  p_k)$,  $Y = (Y_1,  \ldots  , Y_k) \sim {\rm Multi}(n_2, q_1,  \ldots  , q_k)$ be two independent  multinomial distributions with unknown probability vectors $p = (p_1,  \ldots  , p_k)^\top$ and  $ q = (q_1,  \ldots  , q_k)^\top  \in (0,1)^k, $
         respectively ($k \in \mathbb{N}$). We are interested in testing the hypotheses in \eqref{hyp} (with $p^{(0)}$ replaced by $q$), where the distance $d:  \mathbb{R}^k \times \mathbb{R}^k  \to \mathbb{R} $ is defined by a norm on $\mathbb{R}^k $, that is $d(p,q) = \| p - q \| $, which is 
         directionally Hadamard  differentiable in the sense of Definition 2.1 of \cite{fang2019inference}. This includes the  important cases of the $L_1$-distance and maximum deviation distance which will be discussed in Example \ref{ex1}.

In Algorithm \ref{alg1} below  we define a constrained parametric bootstrap test for the hypotheses 
\begin{align}
                \label{hyp1}
                    H_0:  \| p - q \|  \geq \epsilon \quad vs. \quad H_1: \| p - q \|  < \epsilon ~. 
                \end{align}
                For this purpose we denote by
\begin{align}
    \label{d3}
    L(p,q) = \binom{n_1}{x_1, ... , x_k} \binom{n_2}{y_1, ... , y_k} \prod_{i=1}^k p_i^{x_i} q_i^{y_i} 
\end{align}
the likelihood function of  the distribution of the vector $\big ( X^\top , Y^\top \big )^\top $ at $(x^\top,y^\top)^\top$ and denote by $(\hat p , \hat q)$ the corresponding maximum likelihood estimator (MLE). Our main result shows that 
this procedure defines a valid test for the hypotheses \eqref{hyp1}. For this purpose we recall that the function 
$\| \cdot \| : \mathbb{R}^k  \to  \mathbb{R} $ is called {\it directionally} Hadamard  differentiable at the vector $\theta= (p -  q  )^\top \in \R^k $
if  there exists a continuous mapping 
$ d_\theta^\prime    :  \mathbb{R}^k   \to  \mathbb{R} $ such that 
$      \lim_{n \to \infty} \big | \frac{1}{ t_n } ( \| \theta + t_n h_n \|  -  \| \theta \| ) - d_\theta^\prime  ( h )  \big  | = 0$
for all sequences $(h_n)_{n \in \mathbb{N}} \subset \R^k$ with $h_n \to h \in \R^k$  and 
all sequences $(t_n )_{n \in \mathbb{N}} \subset \mathbb{R}^+$ with $t_n \downarrow 0$.  Moreover, if $d_\theta^\prime $  is linear the functional $\| \cdot \|$ is called Hadamard differentiable at the vector $\theta$.

Algorithm \ref{alg1} below defines a bootstrap test for the hypotheses \eqref{hyp1}. Our main result shows that this test  has asymptotic level $\alpha$ and is   consistent. For this purpose we define by $n=n_1 + n_2$ the total sample size and note that a standard argument shows that
$
                    \sqrt{n} \big( ( \hat p - \hat q ) - (p - q) \big) \xrightarrow{d} Z  ,
$
 where we assume  $\lim_{n_1, n_2 \to \infty} n_2/n_1  >0 $, 
 the symbol $\xrightarrow{d} $ means weak convergence and $Z$ denotes 
 a centered $k$-dimensional normal distributed random variable with (degenerate) covariance matrix $\Sigma$.
The delta method for directionally differentiable mappings \citep[see, for example, Theorem 2.1 in][]{shapiro1991asymptotic}
then gives 
 \begin{align}
                    \label{l2}
                       \sqrt{n}  \big ( \| \hat p - \hat q \| - \| p -q \| \big ) \xrightarrow{d}  T:= d_{p-q}^\prime (Z) ~.
\end{align}

\begin{algorithm}[H]
                \small
                \caption{Constrained parametric bootstrap test  }
                \label{alg1}
                    \begin{enumerate}
                        \item Calculate the MLE $ (\hat p, \hat q)$ maximizing \eqref{d3} and   test statistic 
                        $
                            \hat d   =  \|  \hat  p - \hat  q \| . $
                        \item Define the {\it constrained MLE}  $(\tilde{p}, \tilde{q})$ by 
                        $
                          ( \tilde{p}  , \tilde{q}  )^\top  = 
                          \argmax_{
                           \| p-  q \| = \epsilon } L(p,q)
                     $
                        and calculate 
                        \begin{align}
                            (\hat{\hat{p}}, \hat{\hat{q}})^\top  = \begin{cases}
                            (\hat p, \hat q)^\top &  
                            \text{ if } ~\hat d \geq \epsilon \\
                             (\tilde{p}, \tilde{q} )^\top  & \text{ if } ~\hat d < \epsilon.
                            \end{cases} 
                        \end{align}
                        
                        \item For $b=1, \ldots B$  do 
                        \begin{enumerate}
                            \item Generate $X^{* (b) }  \sim {{\rm Multi}}(n_1, \hat{\hat{p}})$, $ Y^{*(b) }  \sim {{\rm Multi}}(n_2, \hat{\hat{q}})$ and the corresponding MLE
                            $({\hat p^{*(b)} },  {\hat q^{*(b) }} )^\top  $.
                            \item Calculate the bootstrap test statistic  
                           $  \hat d^{*(b) } =  \big  \|  \hat p^{*(b) } -  \hat q^{*(b) } \big \|.$
                        \end{enumerate}
                        \item Calculate the empirical $\alpha$-quantile $\hat q_{\alpha}^{*(B)}$ of the bootstrap sample $\hat{d}^{*(1)},  \ldots  , \hat{d}^{*(B)}$
                        and reject $H_0$ whenever
                    \begin{align}
                    \label{test}
                            \hat d < \hat q_{\alpha}^{*(B)}.
                        \end{align}
                    \end{enumerate}
                \end{algorithm}

\begin{Thm}
\label{boot1}
{ \it 
Assume that the distribution of the random variable 
$T$ in \eqref{l2} is continuous and that the 
level  $\alpha$ is sufficiently  small 
such that the corresponding  $\alpha$-quantile   is negative. 
Let  $\hat q_\alpha^*$
denote the $\alpha$-quantile of the distribution of the statistic $\hat d^{*(b) } $. Furthermore,  assume that 
 $\lim_{n_1, n_2 \to \infty} n_2/n_1  >0 $ and that the norm $ \| \cdot \| :
 \mathbb{R}^k \to \mathbb{R}$ is directionally Hadamard  differentiable at $(p- q ) $  with a sublinear derivative.

(a)  If the null hypothesis  in \eqref{hyp1} is satisfied, then
\begin{align} 
\label{d5}
        \limsup_{n_1, n_2 \to \infty} \mathbb P ( \hat d  <  \hat q_{\alpha}^{*} ) \begin{cases}
            = 0,  \text{ if } ~\| p-q\|  > \epsilon, \\
            \leq \alpha, \text{ if } ~ \| p-q\|  = \epsilon.
        \end{cases}
    \end{align}
(b) If  $H_1$ holds, then 
$      \lim_{n_1, n_2 \to \infty} \mathbb P( \hat d <  \hat q_{\alpha}^{*} ) = 1. $
 \smallskip
 
(c) Moreover, if  the norm $ \| \cdot \| $ is  Hadamard  differentiable 
at $p-q$, then \eqref{d5} holds  under the null hypothesis and  additionally $\lim_{n_1, n_2 \to \infty} \mathbb P( \hat d <  \hat q_{\alpha}^{*} ) = \alpha  $,
if $\| p-q\|  = \epsilon $.
} 
\end{Thm}

\begin{ex} \label{ex1}
Theorem \ref{boot1} is applicable to the $L^1$- and $L^\infty $- type 
distances
$d_1 = \| p -q  \|_1 = 
\sum_{i=1}^{k} |p_i - q_i| $
and $
   d_\infty = \| p -q  \|_\infty   = \max_{i=1}^{k} |p_i - q_i|.
    $
In this case $\| \cdot  \|_1$ and $\| \cdot  \|_\infty $ are directionally Hadamard differentiable at any $\theta = p-q \in \R^k$ with derivatives given by
\begin{align}
\label{det13}
               \quad d_{\theta}^\prime (v) =   \sum_{i=1}^{k}  \mathbbm{1}_{ \{ \theta_i \in \mathcal{N}^c \} }  \sgn(\theta_i) v_i + \mathbbm{1}_{ \{  \theta_i \in \mathcal{N} \} } |v_i|~
\end{align}
for $\| \cdot \|_1$, 
where  $ \mathcal{N} := \{ i \in \{1,  \ldots  , k\} \ : \theta_i = 0 \}$, and by  
\begin{align}
\label{det14}
d_{\theta}^\prime (v) =   \max_{i=1}^{k} \{ \max_{i \in \mathcal{E}^{+} } v_i  , \ \max_{i \in \mathcal{E}^{-} } - v_i  \}
\end{align}
for $\| \cdot \|_\infty$,
where $\mathcal{E}^{ \pm } = \{ i \in \{1,  \ldots  , k\} \ : \ \theta_i = \pm d_{\infty} \}$.  
Here \eqref{det13} follows by a straightforward calculation and \eqref{det14} is a consequence of a general result in  \cite{carcamo2020directional}. 
We note that the derivatives are fully Hadamard differentiable at $\theta$ if and only if  $\mathcal{N} = \emptyset$ and the set $\mathcal{E} = \{ i \in \{ 1, ... , k \} : |\theta_i| = d_\infty \}$ consists of exactly one element respectively.

Both derivatives are obviously sublinear. Therefore, they are also convex, and, for both norms, the limiting distribution $T$ defined in \eqref{l2} is a convex functional of a Gaussian variable. Consequently, Corollary 4.4.2 in \cite{bogachev1998gaussian} yields that the corresponding distribution functions are continuous  and Theorem \ref{boot1} is applicable.
\end{ex}

\section{Finite sample properties}

\label{sec3} 
  \def\theequation{3.\arabic{equation}}	
    \setcounter{equation}{0}
  
In this section we empirically verify the theoretical results by means of a small simulation study where we consider the maximum and $L_1$-norm.


We start by considering the maximum norm 
$\|\cdot \|_\infty$. We fix the threshold $\epsilon = 0.25$,  the  significance level $\alpha = 0.05$ and consider a multinomial distribution with  $k=6$ classes.  The two vectors of probabilities are given by  
\begin{align}
\label{det15}
                    p &= ( 1-\delta, \tfrac{\delta}{5}, \ldots  , \tfrac{\delta}{5})^\top ~,~~
                    q = (\delta, \tfrac{1-\delta}{5}, \ldots  , \tfrac{1-\delta}{5})^\top ,
                \end{align}
                where $\delta \in (0,1)$.
                Note that  $\| p - q \|_{\infty} = |1 - 2 \delta|$  and that the maximum is attained at the first class only, such that $|\mathcal{E}| = 1$.
                As pointed out in Section \ref{sec2} the mapping
                $\| \cdot \|_\infty $ is Hadamard differentiable 
                at $p-q$  and part (c) of Theorem \ref{boot1} is applicable in this case. 
                As a second example we  consider the vectors 
                \begin{align}
                \label{det16}
                    p &= ( \tfrac{1-\delta}{2}, \tfrac{1-\delta}{2}, \tfrac{\delta}{4}, \ldots , \tfrac{\delta}{4})^\top  ~,~~
                    q = (\tfrac{\delta}{2}, \tfrac{\delta}{2}, \tfrac{1-\delta}{4}, \ldots  , \tfrac{1-\delta}{4})^\top ,
                \end{align}
                where $\delta \in (0,1)$. In this case  we have $\| p - q \|_\infty  = |0.5 - \delta|$, which  is attained in  the first and second class. Therefore we 
                have  $|\mathcal{E}| = 2$, and we expect the test \eqref{test} to be conservative at the boundary $\|p-q\|_\infty = \epsilon = 0.25$ (see part (a) of Theorem \ref{boot1}).
                The results of this  simulation are displayed in the top row of Figure \ref{fig11}. Overall, the simulation results are consistent with the theoretical findings in Theorem \ref{boot1}. We  observe 
               in both scenarios that for an increasing sample size 
               the rejection probability of the test approaches one under the alternative and zero in the interior of the null hypothesis ($\|p - q \|_\infty > \epsilon$). On the boundary 
               of the  hypotheses ($\|p - q \|_\infty = \epsilon$) some differences between the two scenarios are visible.  In the case \eqref{det15} the significance level  is approximated for an increasing sample size as predicted by part (c) of Theorem \ref{boot1}, see the left upper panel in Figure \ref{fig11}. 
               In the non
               Hadamard differentiable case  corresponding to \eqref{det16} the rejection probability stays below $\alpha$ even for an increasing sample size (see the right upper  panel  and part (a) of Theorem \ref{boot1}).

                In the lower part of  Figure \ref{fig11} we display results for a comparison of the probabilities using the $L^1$-norm.  Here we consider the probability vectors
                \begin{equation}\label{det17}
                \begin{split}
                    p &= ( 0.1, 0.1, 0.1 , 0.2, 0.1, 0.4 )^T \\
                    q &= ( 0.125, 0.125 , 0.125 , 0.125 , 0.1 + \delta , 0.4 - \delta )^T,
                \end{split}
                \end{equation}
                where $\delta \in (0, 0.4)$. In this case, $\| p - q \|_1  = 0.15 + 2 \delta$ and all  coordinates differ such that                 the set in \eqref{det13}  satisfies $\mathcal{N} = \emptyset$. As a second example, we investigate the vectors 
\begin{equation}\label{det18}
                \begin{split}
                    p &= ( 0.1 , 0.1 , 0.1 , 0.2 , 0.1 , 0.4 )^T \\
                    q &= ( 0.1 , 0.125, 0.125, 0.15 , 0.1 + \delta , 0.4 - \delta )^T,
                \end{split}
\end{equation}
                where $\delta \in (0,0.4)$, $\| p - q \|_1  = 0.1 + 2 \delta$. Here the first coordinates are identical such that ${\cal N} =\{ 1 \} $.
        The corresponding rejection probabilities are displayed in the lower part of Figure \ref{fig11}. For an increasing sample size  the rejection probability of the test \eqref{test} increases if $\| p - q \|_1 < \epsilon $  (interior of the alternative)  and decreases if  $\| p - q \|_1 > \epsilon $ (interior of the null hypothesis). Moreover, in the non-differentiable case 
        \eqref{det18} the  test is slightly conservative on the boundary $\| p - q \|_1 = \epsilon $ (right lower panel), while it approximates the nominal level $0.05$  even for small sample sizes in the differentiable case (left lower panel in Figure \ref{fig11}). These results reflect 
        the theoretical findings in Theorem \ref{boot1} (a) and (c), respectively. 
                \begin{figure}[H]
                    \centering
                    \includegraphics[width=7cm, height=4.5cm]{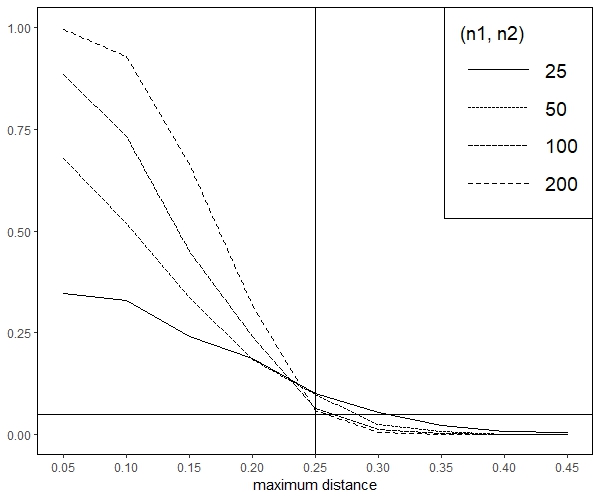}~~~
                    \includegraphics[width=7cm, height=4.5cm]{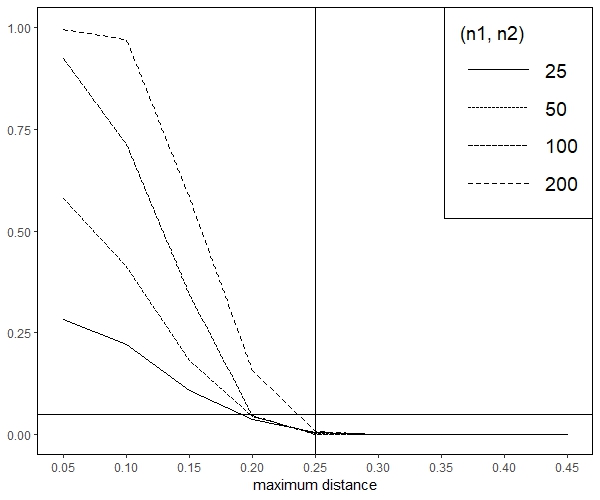}
                    ~~~~ 
                    \includegraphics[width=7cm, height=4.5cm]{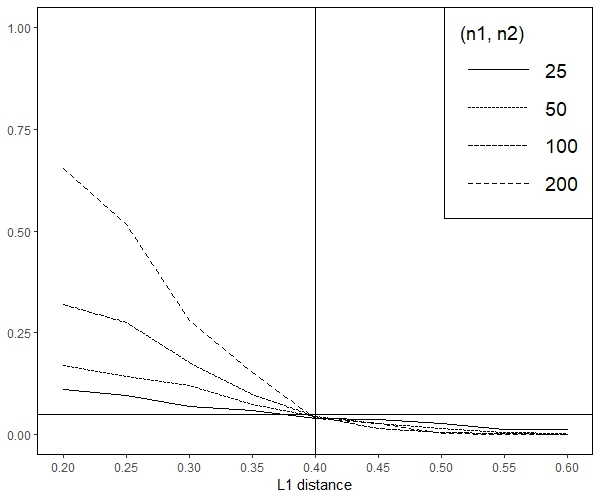}~~~
                    \includegraphics[width=7cm, height=4.5cm]{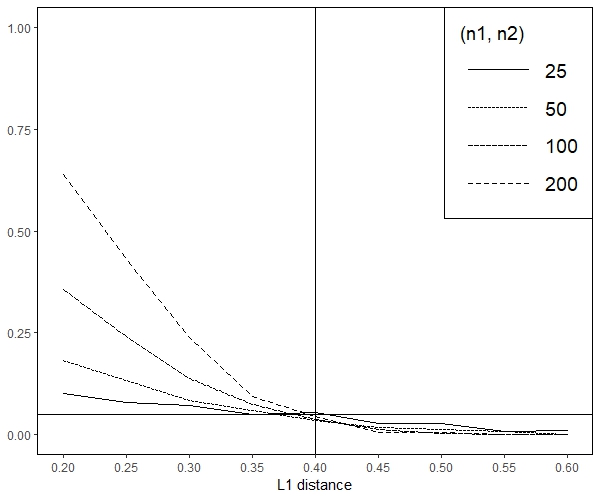}

                    \caption{\it Simulated rejection probabilities of the test \eqref{test}. Upper part: maximum deviation $\| \cdot \|_\infty $ for the probability vectors \eqref{det15} (left) and \eqref{det16} (right).  
                   Lower part: $L^1$-distance $\| \cdot \|_1 $ for the probability vectors \eqref{det17} (left) and \eqref{det18} (right).} 
                    \label{fig11}
                \end{figure}



\section{Appendix:  Proof of  Theorem \ref{boot1}}
        	\label{sec4} 
  \def\theequation{4.\arabic{equation}}	
    \setcounter{equation}{0}

 First, we prove a key lemma that is used for the proof of \eqref{d5}. 
 For this purpose we introduce the notations $\theta = p - q$, $\hat{\hat{\theta}} = \hat{\hat{p}} - \hat{\hat{q}}$, $\hat{\theta}^* = \hat{p}^* - \hat{q}^*$, 
 $\hat d^* = \| \hat \theta^* \|$,  $\hat{\hat{d}} =  \| \hat{\hat{\theta}} \|$ and denote the data by $\mathcal{Y}$.

\begin{Lem}
\label{lem1}
{ \it  Define $\tilde T_n^* := d_{\theta}^\prime  ( \sqrt{n} ( \hat{\theta}^* - \hat{\hat{\theta}} ) ) $, then $\tilde T_n^* \tod T$ conditionally on $\mathcal{Y}$ in probability, and  
    \begin{align}
        \sqrt{n}( \hat d^* - \hat{\hat d} ) \leq \tilde T_n^* + o_{\mathbb P}(1). \label{l10}
    \end{align}
}
\end{Lem}

\begin{proof}[Proof of Lemma \ref{lem1}:]
  We know that 
    \begin{align}
        &\sqrt{n}( \hat{\hat{\theta}} - \theta ) \tod Z ~,~~
        \sqrt{n}( \hat{\theta}^{*} - \hat{\hat{\theta}}) \tod Z \quad \text{conditionally on $\mathcal{Y}$.} \label{l12}
    \end{align}
By the same arguments as given in the proof of Theorem 23.9 in \cite{vaart1998}, we obtain that both sequences $\sqrt{n}( \hat{\hat{\theta}} - \theta )$ and $\sqrt{n}( \hat{\theta}^{*} - \theta )$ converge unconditionally in distribution. Applying the delta method for directionally Hadamard  differentiable functions (see Theorem 2.1 in \cite{fang2019inference}), we therefore obtain for both sequences 
\begin{align*}
     &\sqrt{n}( \| \hat{\hat{\theta}} \| - \| \theta  \| ) = d_\theta^\prime  ( \sqrt{n}( \hat{\hat{\theta}} - \theta ) )  + o_{\mathbb P}(1). \\
     &\sqrt{n}( \| \hat{\theta}^* \| - \| \theta \| ) = d_\theta^\prime  ( \sqrt{n}( \hat{\theta}^* - \theta ) ) + o_{\mathbb P}(1).
\end{align*}
Subtracting the first equation from the second and then using sub-additivity of the directional Hadamard  derivative yields
\begin{align*}
   \sqrt{n}( \hat d^* - \hat{\hat d} ) &= \sqrt{n} ( \| \hat{\theta}^* \| - \| \hat{\hat{\theta}}  \| ) \\
   &= d_{\theta}^\prime  ( \sqrt{n}( \hat{\theta}^* - \theta ) ) - d_\theta^\prime  ( \sqrt{n}( \hat{\theta}^* - \theta ) )  +  o_{\mathbb P}(1) 
    \leq  \tilde T_n^* \ + \ o_{\mathbb P}(1).
\end{align*}
Finally, the latter convergence in \eqref{l12} together with the  continuous mapping theorem
 yield $\tilde T_n^* \tod T$ (note that the directional Hadamard  derivative is continuous, see for example Proposition 3.1 in \cite{shapiro1990concepts} ).
\end{proof}

We continue showing the following two preliminary results that are used for the remainder of the proof:
\begin{align}
    \forall \delta > 0:  \lim_{n \to \infty}  \mathbb P \big( \sqrt{n} ( \hat q_{\alpha}^* - \hat{\hat{d}} ) > q_{\alpha+\delta} \big) = 0. \label{l4} \\
    \lim_{n \to \infty} \mathbb P \big( \hat{\hat{d}} < \hat q_{\alpha}^* \big) = 0, \label{l5}
\end{align}
where $q_{\alpha}$ denotes the $\alpha$-quantile of $T$.

\begin{proof}[Proof of \eqref{l4} and \eqref{l5}]
Let $F_n$ be the distribution function of $\tilde T_n^*$ conditional on the data and $F_n^*$ be the distribution function of $\sqrt{n}( \hat d^* - \hat{\hat d})$ conditional on the data.

 By assumption $q_{\alpha} < 0 $ and  therefore there  exists  a $t_0 < 0$ with $F(t_0) < \alpha$. 
Since $\tilde T_n^* \tod T$ conditionally on $\mathcal{Y}$ in probability, we get 
 $   \sup_{t \geq t_0} |F_n(t) - F(t)| \topr 0.$
From \eqref{l10} we can conclude 
$
  \mathbb P ( F_n^*(t) \geq F_n(t - \epsilon ) \ \forall t) \ton 1
$    for all $\epsilon > 0$.
Since $F$ is uniformly continuous on $[t_0, \infty)$, we therefore obtain
\begin{align}
    \forall \epsilon > 0: \ \mathbb P ( F_n^*(t) \geq F(t) - \epsilon \ \forall t \geq t_0 ) \ton 1. \label{l11}
\end{align} 
Now, let $\delta > 0 $ be arbitrary. Since $F(t_0) < \alpha$, we have $q_{\alpha + \delta} > t_0$. The convergence \eqref{l11} then yields in particular for $t = q_{\alpha + \delta}$ and $\epsilon = \delta/3$ that 
\begin{align}
    \mathbb P( F_n^*(q_{\alpha + \delta}) < F(q_{\alpha + \delta}) - \delta / 3 ) = \mathbb P( F_n^*(q_{\alpha + \delta}) < \alpha + (2/3)\delta  ) \ton 0.
\end{align}
Therefore, we can conclude that
\begin{align*}
    \mathbb P( \sqrt{n} ( \hat q_{\alpha}^* - \hat{\hat{d}} ) > q_{\alpha + \delta}  ) &\leq \mathbb P ( F_n^*( \sqrt{n} ( \hat q_{\alpha}^* - \hat{\hat{d}} ) ) \geq F_n^* ( q_{\alpha + \delta} )  ) \\
    &= \mathbb P ( \alpha \geq F_n^*(q_{\alpha + \delta} ) )  \leq \mathbb P( F_n^*(q_{\alpha+\delta}) < \alpha + (2/3)\delta )  = o(1),
\end{align*}
which proves \eqref{l4}.  For a proof of \eqref{l5} note that we  have assumed $q_{\alpha} < 0$. Therefore   we can choose $\delta > 0$ small enough such that $q_{\alpha + \delta} < 0$ as well which gives, observing  \eqref{l4}, 
    \begin{align*}
        \mathbb P( \hat{\hat{d}} < \hat q_{\alpha}^* ) = \mathbb P ( \sqrt{n}( \hat q_{\alpha}^* - \hat{\hat{d}} ) > 0 ) \leq \mathbb P ( \sqrt{n}( \hat q_{\alpha}^* - \hat{\hat{d}} ) > q_{\alpha + \delta} ) \ton 0.
    \end{align*}
\end{proof}

Using \eqref{l4} and \eqref{l5}, we can now complete the proof of  Theorem \ref{boot1} starting with part (a). 
If  \underline{\textit{ $d > \epsilon$}}, 
 we use
\begin{align*}
    \mathbb P( \hat d < \hat q_{\alpha}^* ) = \mathbb P( \hat d < \hat q_{\alpha}^* , \ \hat d \geq \epsilon ) + \mathbb P( \hat d < \hat q_{\alpha}^* , \ \hat d < \epsilon ).
\end{align*}
The first term satisfies
$
    \mathbb P( \hat d < \hat q_{\alpha}^* , \ \hat d \geq \epsilon ) = \mathbb P( \hat d < \hat q_{\alpha}^*, \ \hat d = \hat{\hat{d}} ) \leq \mathbb P ( \hat{\hat{d}} < \hat{q}_{\alpha}^* ) \ton 0,
$
by  \eqref{l5}. The second 
term satisfies
\begin{align*}
    \mathbb P( \hat d < \hat q_{\alpha}^* , \ \hat d < \epsilon ) \leq \mathbb P ( \hat d < \epsilon ) = \mathbb P ( \sqrt{n}(\hat d - d ) < \sqrt{n} ( \epsilon - d ) ) \ton 0,
\end{align*}
since $\sqrt{n}(\hat d - d)$ converges weakly by the directional delta method and $\sqrt{n}(\epsilon - d) \ton - \infty$ by the assumption $d > \epsilon$. All in all, this gives 
$
   \lim_{n \to \infty} \mathbb P( \hat d < \hat q_{\alpha}^* ) = 0
$, if $d > \epsilon$.

In the case  \underline{\textit{$d = \epsilon $}} we use  
\begin{align*}
    \mathbb P ( \hat d < \hat q_{\alpha}^* ) = \mathbb P( \hat d < \hat q_{\alpha}^* , \ \hat{\hat{d}} = \epsilon ) + \mathbb P( \hat d < \hat q_{\alpha}^* , \ \hat{\hat{d}} > \epsilon ). 
\end{align*}
To deal with  the first term we introduce  the notation 
$ \hat p_{\alpha}^*:= \sqrt{n}( \hat q_{\alpha}^* - \hat{\hat{d}} )$
and obtain for any $\delta > 0$
\begin{align}
\nonumber
    \mathbb P( \hat d < \hat q_{\alpha}^* , \ \hat{\hat{d}} = \epsilon ) &= \mathbb P( \sqrt{n} ( \hat d - d ) < \sqrt{n}( \hat q_{\alpha}^* - \hat{\hat{d}} ) , \ \hat{\hat{d}} = \epsilon ) \\
    &\leq \mathbb P( \sqrt{n} ( \hat d - d ) < \sqrt{n}( \hat q_{\alpha}^* - \hat{\hat{d}} )
    ) \nonumber \\
    &= \mathbb P ( \sqrt{n} ( \hat d - d ) < \hat p_{\alpha}^{*}, \ \hat p_{\alpha}^{*} \leq q_{\alpha + \delta} ) + \mathbb P ( \sqrt{n} ( \hat d - d ) < \hat p_{\alpha}^{*}, \ \hat p_{\alpha}^{*} > q_{\alpha + \delta} )~.  \label{l7}
\end{align}
 For the first sequence in \eqref{l7}, we have
\begin{align}
    \mathbb P ( \sqrt{n} ( \hat d - d ) < \hat p_{\alpha}^{*}, \ \hat p_{\alpha}^{*} \leq q_{\alpha + \delta} ) \leq \mathbb P( \sqrt{n} ( \hat d - d ) \leq q_{\alpha + \delta} ) \ton \alpha + \delta, \label{l8}
\end{align}
by the weak convergence of $\sqrt{n}( \hat d - d )$ and since the limit distribution function is continuous at $q_{\alpha + \delta}$.  For  the second term in \eqref{l7}, we have
\begin{align}
    \mathbb P ( \sqrt{n} ( \hat d - d ) < \hat p_{\alpha}^{*}, \ \hat p_{\alpha}^{*} > q_{\alpha + \delta} ) \leq \mathbb P ( \hat p_{\alpha}^* > q_{\alpha + \delta} ) \ton 0, \label{l9}
\end{align}
by \eqref{l4}. Putting \eqref{l8} and \eqref{l9} together, we obtain
$
    \limsup_{n \to \infty} \mathbb P( \hat d < \hat q_{\alpha}^* ) \leq \alpha + \delta~
$
 for any $\delta > 0$. This 
is equivalent to $\limsup_{n \to \infty} \mathbb P( \hat d < \hat q_{\alpha}^* ) \leq \alpha$, which completes the proof of part (a). 
\medskip 

The proof of part (b) follows by similar arguments as given in the proof of Theorem 2 in \cite{dette2018equivalence}  (note that the map $(p,q) \to \| p - q \| $ is uniformly continuous). 
\medskip 

For a proof of part (c), we first determine the asymptotic distribution of the bootstrap test statistic $\hat d^*$. To this end, recall that the second convergence in \eqref{l12} holds, that is $\sqrt{n}( \hat \theta^* - \hat{\hat \theta} ) \tod Z$ conditionally on $\mathcal{Y}$ in probability. Since we assume that the directional Hadamard derivative of the norm $ \| \cdot \|$ 
is linear, we can apply the bootstrap delta method for Hadamard  differentiable functions as stated in Theorem 23.9 in \cite{vaart1998} in order to obtain
\begin{align}
\label{l1}
    \sqrt{n} ( \| \hat \theta^{*} \| - \| \hat{\hat{\theta}} \| ) \tod  d_{\theta}^\prime (Z)
\end{align}
conditionally on $\mathcal{Y}$ in probability. Employing the convergences \eqref{l2} and \eqref{l1}, we can now proceed with exactly the same arguments as given in the proof of Theorem 3.3 a) cases 1 and 2 in \cite{dette2018equivalence} in order to finish the proof.

\bigskip\noindent
\textbf{Acknowledgements:}  
This research is supported by the European Union through the European Joint Programme on Rare Diseases under the European Union's Horizon 2020 Research and Innovation Programme Grant Agreement Number 825575.

    \bibliographystyle{apalike}
    \setlength{\bibsep}{2pt}
    \bibliography{bibliography.bib}

\begin{thebibliography}{}

\bibitem[Alba-Fern{\'a}ndez and Jim{\'e}nez-Gamero, 2023]{Alba-Fernndez2023}
Alba-Fern{\'a}ndez, M.~V. and Jim{\'e}nez-Gamero, M.~D. (2023).
\newblock {\em Equivalence tests for multinomial data based on
  $\phi$-divergences}, pages 121--129.
\newblock Springer International Publishing, Cham.

\bibitem[Bogachev, 1998]{bogachev1998gaussian}
Bogachev, V.~I. (1998).
\newblock {\em Gaussian measures}.
\newblock Number~62. American Mathematical Society.

\bibitem[C{\'a}rcamo et~al., 2020]{carcamo2020directional}
C{\'a}rcamo, J., Cuevas, A., and Rodr{\'i}guez, L.-A. (2020).
\newblock {Directional differentiability for supremum-type functionals:
  Statistical applications}.
\newblock {\em Bernoulli}, 26(3):2143 -- 2175.

\bibitem[Dette et~al., 2018]{dette2018equivalence}
Dette, H., M{\"o}llenhoff, K., Volgushev, S., and Bretz, F. (2018).
\newblock Equivalence of regression curves.
\newblock {\em Journal of the American Statistical Association},
  113(522):711--729.

\bibitem[Fang and Santos, 2019]{fang2019inference}
Fang, Z. and Santos, A. (2019).
\newblock Inference on directionally differentiable functions.
\newblock {\em The Review of Economic Studies}, 86(1):377--412.

\bibitem[Frey, 2009]{frey2009exact}
Frey, J. (2009).
\newblock An exact multinomial test for equivalence.
\newblock {\em Canadian Journal of Statistics}, 37(1):47--59.

\bibitem[Ostrovski, 2017]{ostr2017}
Ostrovski, V. (2017).
\newblock Testing equivalence of multinomial distributions.
\newblock {\em Statistics \& Probability Letters}, 124:77--82.

\bibitem[Ostrovski, 2018]{OSTROVSKI201861}
Ostrovski, V. (2018).
\newblock Testing equivalence to families of multinomial distributions with
  application to the independence model.
\newblock {\em Statistics \& Probability Letters}, 139:61--66.

\bibitem[Shapiro, 1990]{shapiro1990concepts}
Shapiro, A. (1990).
\newblock On concepts of directional differentiability.
\newblock {\em Journal of optimization theory and applications}, 66:477--487.

\bibitem[Shapiro, 1991]{shapiro1991asymptotic}
Shapiro, A. (1991).
\newblock Asymptotic analysis of stochastic programs.
\newblock {\em Annals of Operations Research}, 30(1):169--186.

\bibitem[Van~der Vaart, 1998]{vaart1998}
Van~der Vaart, A.~W. (1998).
\newblock {\em Asymptotic Statistics. Cambridge Series in Statistical and
  Probabilistic Mathematics}.
\newblock Cambridge University Press, Cambridge.

\bibitem[Wellek, 2010]{wellek2010testing}
Wellek, S. (2010).
\newblock {\em Testing statistical hypotheses of equivalence and
  noninferiority}.
\newblock CRC Press.

\end{thebibliography}

\end{document}